\documentclass[11pt]{article}

\usepackage[a4paper,margin=1in]{geometry}
\usepackage[T1]{fontenc}
\usepackage{lmodern}
\usepackage{microtype}
\usepackage{amsmath,amssymb,amsthm,mathtools}
\usepackage{hyperref}
\hypersetup{
	colorlinks=true,
	linkcolor=blue,
	filecolor=magenta,
	urlcolor=cyan,
	citecolor=red,
	pdftitle={A higher-connectivity spectral Ore theorem for triangle-free graphs},
	pdfsubject={Extremal spectral graph theory}
}
\newtheorem{theorem}{Theorem}[section]
\newtheorem{lemma}[theorem]{Lemma}
\newtheorem{proposition}[theorem]{Proposition}
\newtheorem{corollary}[theorem]{Corollary}

\theoremstyle{definition}

\theoremstyle{remark}

\newtheorem{problem}[theorem]{Problem}

\newcommand{\V}{V}
\newcommand{\E}{E}
\newcommand{\e}{e}
\newcommand{\rhoA}{\rho}
\newcommand{\comp}[1]{\overline{#1}}
\newcommand{\floor}[1]{\left\lfloor #1\right\rfloor}
\newcommand{\ceil}[1]{\left\lceil #1\right\rceil}

\title{A higher-connectivity spectral Ore theorem for triangle-free graphs}

\author{Joyentanuj Das\footnote{Department of Mathematics, College of Engineering and Technology, SRM Institute of Science and Technology, Kattankulathur, Chennai, 603203, India. \indent  Emails: joyentanuj@gmail.com,  joyentad@srmist.edu.in.} \quad and \quad Sayan Gupta\footnote{School of Mathematical Sciences, NISER Bhubaneswar (An OCC of Homi Bhabha National Institute, Mumbai, 400094), India. Email: sayan.gupta@niser.ac.in}}
\date{}

\begin{document}
\maketitle

\begin{abstract}
Let $B_{n,k}$ be the graph obtained from the balanced complete bipartite
graph on $n$ vertices by deleting a matching of size $k$.  If $G$ is an
$n$-vertex triangle-free graph with $\kappa(\comp G)\geq k$, we prove that
$\rhoA(G)\leq\rhoA(B_{n,k})$ for $n\geq4k+2$, with equality precisely when
$G\cong B_{n,k}$, and we compute $\rhoA(B_{n,k})$ explicitly.  We also
solve the bipartite problem for every $n\geq2k+1$, determine the boundary
value $\operatorname{spex}_{\kappa}(2k,K_3;k)=k-1$, and settle the full
problem for $k=2$.  In particular, $B_{n,2}$ is uniquely extremal exactly
from order $6$ onward.  For $k=1$, equivalently when the complement is
connected, $B_{n,1}=K_{\ceil{n/2},\floor{n/2}}-e$ is uniquely extremal for
every $n\geq3$.
\end{abstract}



\section{Introduction and literature survey}

All graphs in this paper are finite, simple, and undirected.  The starting
point is Mantel's theorem \cite{Mantel1907}, which asserts that a
triangle-free graph of order $n$ has at most $\floor{n^2/4}$ edges, with
equality only for the balanced complete bipartite graph.  Tur\'an
\cite{Turan1941} extended this result from triangles to complete graphs:
among all $K_{r+1}$-free graphs of order $n$, the balanced complete
$r$-partite graph $T_r(n)$ is the unique graph of maximum size.  These
theorems are foundational examples of extremal graph theory.

Ore \cite[p.~216]{Ore1962} asked for the maximum number of edges in an
$n$-vertex $K_{r+1}$-free graph whose complement is connected.  After
taking complements, this is equivalent to minimizing the number of edges
in a connected graph with independence number at most $r$.  The edge
problem was solved independently by Christophe et al.
\cite{ChristopheEtAl2008} and Gitler and Valencia
\cite{GitlerValencia2010}; a short proof and a structural treatment were
subsequently given by Bougard and Joret \cite{BougardJoret2008}.  In the
original language, the maximum is
\[
 \e(T_r(n))-r+1.
\]
The complement condition matters: $\comp{T_r(n)}$ is a disjoint union of
$r$ cliques and is therefore disconnected when $r\geq2$.

The adjacency spectral radius $\rhoA(G)$ is a natural substitute for the
number of edges.  Nikiforov's inequality \cite{Nikiforov2002} implies that
a $K_{r+1}$-free graph with $m$ edges satisfies
\[
 \rhoA(G)^2\leq 2\left(1-\frac1r\right)m,
\]
and his spectral Tur\'an theorem \cite{Nikiforov2007} identifies $T_r(n)$
as the unique $K_{r+1}$-free graph of maximum spectral radius.  Surveys of
this circle of ideas include \cite{LiLiuFeng2022,Nikiforov2011}.  Liu and
Ning \cite{LiuNing2026} recently established a spectral version of Ore's
problem: for $n$ sufficiently large relative to $r$, the unique
$K_{r+1}$-free graph of largest spectral radius among graphs with connected
complement is obtained from $T_r(n)$ by deleting a star on $r$ vertices.
Their theorem holds for $n\geq65r^2$ and explicitly raises the problem of
improving the order hypothesis and of treating stronger connectivity
requirements.  When $r=2$, their construction is
$K_{\ceil{n/2},\floor{n/2}}-e$ and their order hypothesis becomes
$n\geq260$.  Corollary~\ref{cor:k-one} below removes this order
restriction completely, proving the same conclusion for every $n\geq3$.
More generally, Theorem~\ref{thm:main} replaces connectedness of the
complement by $k$-connectivity and determines the unique extremal graph
when $n\geq4k+2$.  The additional exact results in Section~4 solve the
bipartite problem throughout its natural range $n\geq2k+1$, determine the
boundary value at $n=2k$, and solve the unrestricted problem for $k=2$ at
every admissible order.

Higher connectivity already has a substantial edge-extremal history.
Bougard and Joret \cite{BougardJoret2008} studied the minimum size
$f(n,\alpha,k)$ of a $k$-connected graph of order $n$ and independence
number $\alpha$.  They settled the connected and $2$-connected cases,
proved several structural refinements of Brouwer's theorem
\cite{Brouwer1981}, and formulated a general conjecture for $k\geq3$.
They also verified the conjectured value for $\alpha=2$ and $n\geq2k$.
In complementary language, $\alpha=2$ is precisely the triangle-free
case considered here.

Triangle-free graphs have also served as a testing ground for sharper
spectral phenomena.  The spectral Mantel inequality
$\rhoA(G)^2\leq \e(G)$ is contained in \cite{Nikiforov2002} and goes back
to Nosal \cite{Nosal1970}.  Lin, Ning, and Wu \cite{LinNingWu2021}
determined the maximum
spectral radius of a non-bipartite triangle-free graph of prescribed
order.  They also proved that every such graph with $m$ edges satisfies
$\rhoA(G)\leq\sqrt{m-1}$, improving the general bound
$\rhoA(G)\leq\sqrt m$; equality holds precisely when $G$ consists of a
$5$-cycle together with possibly some isolated vertices.  Zhai and Shu
\cite{ZhaiShu2022} refined the fixed-size result, and Li, Feng, and Peng
\cite{LiFengPeng2024} completed the remaining even-size extremal problem.
In another recent direction, Hou, Xie, and Fang
\cite{HouXieFang2026} proved that a connected triangle-free graph of
order $n$ and diameter three has spectral radius at most that of
$K_{\ceil{n/2},\floor{n/2}}-e$.  Their result concerns the subclass of
connected triangle-free graphs having diameter three.  By contrast,
Corollary~\ref{cor:k-one} assumes only that the complement is connected
and permits disconnected graphs as well as connected graphs of arbitrary
diameter.
These results show that forcing non-bipartiteness creates a genuine
spectral deficit, just as in the classical edge setting.

Motivated by these developments, define
\[
 \operatorname{spex}_{\kappa}(n,K_3;k)
 :=\max\{\rhoA(G): |\V(G)|=n,\ K_3\nsubseteq G,
                    \ \kappa(\comp G)\geq k\}.
\]
Our main result gives this parameter exactly when $n\geq4k+2$ and
characterizes the unique extremal graph: one deletes a $k$-matching from
the balanced complete bipartite graph.  An exact quotient calculation
gives the candidate's spectral radius.  The spectral Mantel inequality
and the sharp non-bipartite edge bound force a large-order extremizer to be
bipartite.  In a bipartition, $k$-connectivity of the complement and
K\H{o}nig's theorem force a $k$-matching among the missing cross-edges;
Perron--Frobenius monotonicity and the quotient formula complete the
balancing and uniqueness arguments.  Combining this bipartite reduction
with the order-based theorem of Lin, Ning, and Wu yields the complete
$k=2$ classification.

The paper has five sections.  Section~2 records notation and gives
self-contained proofs of the density and spectral inequalities used later.
Section~3 analyzes the proposed extremal construction.  Section~4 proves
the main theorem, establishes the exact small-connectivity and bipartite
results, and derives an asymptotic expansion.  Section~5 lists the natural
remaining problems.

\section{Notation and preliminary results}

For a graph $G$, write $\V(G)$ and $\E(G)$ for its vertex and edge sets,
$\e(G)=|\E(G)|$ for its size, and $\comp G$ for its complement.  The
adjacency matrix of $G$ is denoted by $A(G)$, and its largest eigenvalue by
$\rhoA(G)$.  If $G$ is connected, a positive unit eigenvector belonging to
$\rhoA(G)$ is called a Perron vector.  We write $K_{p,q}$ for the complete
bipartite graph with part sizes $p$ and $q$, and $M_k$ for a matching with
$k$ edges.

The vertex-connectivity $\kappa(H)$ is the minimum number of vertices whose
deletion disconnects $H$ or leaves a single vertex; as usual,
$\kappa(K_n)=n-1$.  A graph is $k$-connected if it has at least $k+1$
vertices and remains connected after the deletion of fewer than $k$
vertices.  For a graph $F$, let $\tau(F)$ and $\nu(F)$ denote its minimum
vertex-cover number and maximum matching number, respectively.  We shall
use K\H{o}nig's bipartite matching theorem \cite{Konig1931} in the form
\begin{equation}\label{eq:konig}
 \tau(F)=\nu(F)\qquad\text{for every bipartite graph }F.
\end{equation}

We begin with two classical edge estimates that drive the structural
reduction.  The first is Mantel's theorem \cite{Mantel1907}; the second is
the sharp non-bipartite refinement due to Erd\H{o}s \cite{Erdos1962}.
Their proofs are omitted because both results are classical.

\begin{lemma}\label{lem:density}
Let $G$ be a triangle-free graph of order $n$ and size $m$.
\begin{enumerate}
\item[(i)] One has $m\leq\floor{n^2/4}$, with equality if and only if
$G\cong K_{\floor{n/2},\ceil{n/2}}$.
\item[(ii)] If $G$ is non-bipartite, then
\[
 m\leq \floor{\frac{(n-1)^2}{4}}+1.
\]
\end{enumerate}
\end{lemma}

The bound in part~\textup{(ii)} is sharp.  For every $n\geq5$, equality is
attained by subdividing one edge of
\[
 K_{\floor{(n-1)/2},\ceil{(n-1)/2}}.
\]
Indeed, subdivision replaces one edge by two edges, so the resulting
triangle-free non-bipartite graph has $\floor{(n-1)^2/4}+1$ edges.  For
$b=\floor{n/2}$ and $s=\floor{n^2/4}$, a direct parity check gives
\begin{equation}\label{eq:nonbip-threshold}
 \floor{\frac{(n-1)^2}{4}}+1=s-b+1.
\end{equation}
Thus every triangle-free graph with at least $s-b+2$ edges is bipartite.
This is the $r=2$ instance of the stability lemma (Lemma 2.1) used in
\cite{LiuNing2026}.

The next inequality is usually called Nosal's spectral Mantel inequality.
It was proved by Nosal \cite{Nosal1970}; the formulation used here also
follows from Nikiforov's more general clique-number inequality
\cite{Nikiforov2002}.

\begin{lemma}\label{lem:nosal}
If $G$ is triangle-free with $m$ edges, then
\[
 \rhoA(G)^2\leq m.
\]
\end{lemma}

We will also use the following standard strict form of spectral
monotonicity, which is an immediate consequence of the
Perron--Frobenius comparison theorem; see, for example,
\cite[Chapter~8]{HornJohnson2013}.

\begin{lemma}\label{lem:pf}
If $H$ is connected and $G$ is a proper spanning subgraph of $H$, then
$\rhoA(G)<\rhoA(H)$.
\end{lemma}

\begin{proof}
The matrices satisfy $0\leq A(G)\leq A(H)$ and are unequal, while $A(H)$
is irreducible because $H$ is connected.  The strict comparison theorem
for nonnegative matrices states that if $0\leq X\leq Y$, $X\ne Y$, and
$Y$ is irreducible, then $\rho(X)<\rho(Y)$.  Applying it with
$X=A(G)$ and $Y=A(H)$ gives $\rhoA(G)<\rhoA(H)$.
\end{proof}

\section{The extremal construction and its spectrum}

Let $p,q,k$ be positive integers with $p,q\geq k$.  Define
\[
 B_{p,q}^{(k)}:=K_{p,q}-M_k,
\]
where the deleted edges form a matching.  Up to isomorphism this graph is
independent of the chosen matching.  For $n\geq2k$, put
\[
 a=\ceil{n/2},\qquad b=\floor{n/2},\qquad
 B_{n,k}:=B_{a,b}^{(k)}.
\]
When it exists, let $n_0(k)$ denote the least integer $N$ such that
$B_{n,k}$ is the unique extremal graph defining
$\operatorname{spex}_{\kappa}(n,K_3;k)$ for every $n\geq N$.

\begin{proposition}\label{prop:connectivity}
If $p,q>k$, then
\[
 \kappa\bigl(\comp{B_{p,q}^{(k)}}\bigr)=k.
\]
\end{proposition}

\begin{proof}
Let $P,Q$ be the two bipartition classes, and label the deleted matching
as
\[
 M_k=\{x_i y_i:1\leq i\leq k\},
 \qquad x_i\in P,\quad y_i\in Q.
\]
Put $H=\comp{B_{p,q}^{(k)}}$.  Since $B_{p,q}^{(k)}$ has no edges inside
$P$ or inside $Q$, both $P$ and $Q$ induce cliques in $H$.  The only
edges of $H$ between these two cliques are the matching edges
$x_i y_i$, $1\leq i\leq k$.

We first prove $\kappa(H)\geq k$.  Let $S\subseteq\V(H)$ with
$|S|\leq k-1$.  Because the edges $x_i y_i$ are pairwise vertex-disjoint,
each vertex of $S$ is an endpoint of at most one matching edge.  Thus
$S$ can meet at most $|S|\leq k-1$ of the $k$ matching edges, and there is
an index $i$ for which
\[
 x_i,y_i\notin S.
\]
In particular, the cross-edge $x_i y_i$ survives in $H-S$.  Moreover,
$p,q>k>|S|$ implies that $P\setminus S$ and $Q\setminus S$ are both
nonempty, and they still induce cliques.  If $u\in P\setminus S$ and
$v\in Q\setminus S$, then
\[
 u,x_i,y_i,v
\]
is a walk in $H-S$ (with the obvious shortening if $u=x_i$ or
$v=y_i$).  Vertices in the same one of the two parts are adjacent.
Consequently every pair of vertices of $H-S$ is joined by a path, so
$H-S$ is connected.  Since this holds for every set $S$ of fewer than
$k$ vertices, $\kappa(H)\geq k$.

For the reverse inequality, let
\[
 X=\{x_1,\ldots,x_k\}\subseteq P.
\]
Since $p>k$, the set $P\setminus X$ is nonempty; the set $Q$ is also
nonempty.  Every cross-edge of $H$ has its endpoint in $P$ belonging to
$X$, so $H-X$ has no edge between $P\setminus X$ and $Q$.  These two
nonempty sets therefore lie in different components of $H-X$.  Hence
$X$ is a vertex cut of size $k$, proving $\kappa(H)\leq k$.  Combining
the two inequalities gives $\kappa(H)=k$.
\end{proof}

We next calculate the spectral radius exactly.

\begin{proposition}\label{prop:spectrum}
Let $p+q=n$ and $p,q>k$.  Then
\begin{equation}\label{eq:spectrum-general}
 \rhoA\bigl(B_{p,q}^{(k)}\bigr)^2
 =\frac{pq-2k+1+
 \sqrt{(pq-2k+1)^2-4(p-k)(q-k)}}{2}.
\end{equation}
For fixed $n$ and $k$, this quantity is strictly increasing as a function
of $pq$.  Consequently,
\[
 \rhoA\bigl(B_{p,q}^{(k)}\bigr)
 \leq\rhoA(B_{n,k}),
\]
with equality if and only if $\{p,q\}=\{a,b\}$.
\end{proposition}

\begin{proof}
Write $J_{r,s}$ for the $r$-by-$s$ all-ones matrix, put
$J_r=J_{r,r}$, and let $I_r$ denote the identity matrix of order $r$.
Label the deleted matching as $x_i y_i$, $1\leq i\leq k$, and order the
two parts as
\[
 P=(x_1,\ldots,x_k;P_0), \qquad  Q=(y_1,\ldots,y_k;Q_0),
\]
where $|P_0|=p-k$ and $|Q_0|=q-k$.  With respect to this ordering, the
bipartite adjacency matrix has the form
\[
 A\bigl(B_{p,q}^{(k)}\bigr)
 =\begin{pmatrix}0&C\\ C^{\mathsf T}&0\end{pmatrix},
\]
where the explicit block form of the $p$-by-$q$ matrix $C$ is
\begin{equation}\label{eq:C-block}
 C=
 \begin{pmatrix}
  J_k-I_k & J_{k,q-k}\\
  J_{p-k,k} & J_{p-k,q-k}
 \end{pmatrix}.
\end{equation}
Indeed, the only zero entries of $C$ are those corresponding to the
deleted pairs $x_i y_i$; all other pairs with one endpoint in each part
remain adjacent.

If $(x,y)^{\mathsf T}$ is an eigenvector with
eigenvalue $\lambda$, then $Cy=\lambda x$ and
$C^{\mathsf T}x=\lambda y$, so $CC^{\mathsf T}x=\lambda^2x$.
Equivalently, the nonzero eigenvalues of the displayed adjacency matrix
are the positive and negative singular values of $C$.  Hence
$\rhoA(B_{p,q}^{(k)})^2$ is the largest eigenvalue of
$CC^{\mathsf T}$.

We next compute this product explicitly.  Two distinct rows among the
first $k$ rows of $C$ have scalar product $q-2$, because their zero
entries occur in different columns.  Each of these rows has squared norm
$q-1$.  The scalar product of one of the first $k$ rows with an all-ones
row is $q-1$, while the scalar product of any two all-ones rows is $q$.
It follows that
\begin{equation}\label{eq:CCT-block}
 CC^{\mathsf T}
 =\begin{pmatrix}
  (q-2)J_k+I_k & (q-1)J_{k,p-k}\\
  (q-1)J_{p-k,k} & qJ_{p-k}
 \end{pmatrix}.
\end{equation}

We record carefully why $CC^{\mathsf T}$ is irreducible.  Since $q>k$,
we have $q-1>0$.  If $k\geq2$, then $q\geq k+1\geq3$, so also
$q-2>0$.  If $k=1$, the first diagonal block has order one and hence has
no off-diagonal entries; its unique entry is $q-1>0$.  Since also
$p-k>0$ and $q>0$, every entry in the block matrix
\eqref{eq:CCT-block} is strictly positive.  Thus
\[
 CC^{\mathsf T}>0
\]
entrywise.  The directed graph associated with a positive matrix has an
arc between every ordered pair of vertices and is therefore strongly
connected.  Equivalently, no simultaneous permutation of the rows and
columns can put a positive matrix into a proper block upper-triangular
form.  Hence $CC^{\mathsf T}$ is irreducible.

Partition its indices into the first $k$ rows and the remaining $p-k$
rows.  The block form \eqref{eq:CCT-block} shows that this is an equitable
partition.  Its quotient matrix is
\begin{equation}\label{eq:quotient}
 Q=\begin{pmatrix}
 kq-2k+1 &(p-k)(q-1)\\
 k(q-1)  &(p-k)q
 \end{pmatrix}.
\end{equation}
Indeed, the row sums from an index in the first class into the two classes
are
\[
 (q-1)+(k-1)(q-2)=kq-2k+1
 \quad\text{and}\quad (p-k)(q-1).
\]
For an index in the second class, the corresponding sums are $k(q-1)$ and
$(p-k)q$, giving \eqref{eq:quotient}.  Permutations within either index
class preserve $CC^{\mathsf T}$.  Since this matrix is irreducible, its
positive Perron eigenvector is unique up to scalar multiplication and
must therefore be fixed by all such permutations.  It is consequently
constant on each class.  Substitution of these two constant values shows
that its Perron eigenvalue is an eigenvalue of $Q$.  Conversely, $Q$ is a
positive $2$-by-$2$ matrix, and lifting its positive Perron eigenvector to
a vector constant on the two classes produces a positive eigenvector of
$CC^{\mathsf T}$.  Therefore $Q$ and $CC^{\mathsf T}$ have the same
largest eigenvalue.

A direct calculation gives
\[
 \operatorname{tr}Q=pq-2k+1,
\]
and
\begin{align*}
 \det Q
 &=(p-k)\bigl(q(kq-2k+1)-k(q-1)^2\bigr)\\
 &=(p-k)(q-k).
\end{align*}
Thus the largest eigenvalue of $Q$ is the larger root of
\begin{equation}\label{eq:quadratic}
 z^2-(pq-2k+1)z+(p-k)(q-k)=0,
\end{equation}
which proves \eqref{eq:spectrum-general}.

It remains to balance the parts.  Write $s=pq$.  Since $p+q=n$,
\[
 (p-k)(q-k)=s-kn+k^2.
\]
On the real interval spanned by the feasible values of $pq$, let $\mu(s)$
be the larger root of
\begin{equation}\label{eq:mu-s}
 \mu^2-(s-2k+1)\mu+s-kn+k^2=0.
\end{equation}
If $f_s(z)$ denotes the polynomial on the left of \eqref{eq:mu-s}, then
\[
 f_s(1)=k(k+2-n)<0,
\]
because $p,q>k$ implies $n\geq2k+2$.  Since $f_s$ is monic, it therefore
has two distinct real roots, one smaller than $1$ and one larger than $1$.
In particular, $\mu(s)>1$ throughout the relevant interval.
Differentiating \eqref{eq:mu-s} with respect to $s$ gives
\[
 2\mu\mu'-\mu-(s-2k+1)\mu'+1=0,
\]
and hence
\[
 \mu'(s)
 =\frac{\mu(s)-1}{2\mu(s)-(s-2k+1)}.
\]
Let $\eta(s)$ denote the smaller root of \eqref{eq:mu-s}.  By Vieta's
formula,
\[
 \mu(s)+\eta(s)=s-2k+1.
\]
Consequently, the denominator in the expression for $\mu'(s)$ is
\[
 2\mu(s)-(s-2k+1)
 =2\mu(s)-\bigl(\mu(s)+\eta(s)\bigr)
 =\mu(s)-\eta(s).
\]
If
\[
 \Delta(s)=(s-2k+1)^2-4(s-kn+k^2)
\]
is the discriminant of \eqref{eq:mu-s}, then the quadratic formula gives
\[
 \mu(s)=\frac{s-2k+1+\sqrt{\Delta(s)}}2,
 \qquad
 \eta(s)=\frac{s-2k+1-\sqrt{\Delta(s)}}2.
\]
Thus $\mu(s)-\eta(s)=\sqrt{\Delta(s)}>0$.  The numerator
$\mu(s)-1$ is also positive, so $\mu'(s)>0$.  Finally,
\[
 pq=p(n-p)\leq\floor{n^2/4}=ab,
\]
with equality exactly when the two integer part sizes differ by at most
one, that is, when $\{p,q\}=\{a,b\}$.  Strict increase of $\mu$ proves the
balancing assertion and its equality condition.
\end{proof}

For later use we locate the two roots of \eqref{eq:quadratic} in the
balanced case.

\begin{lemma}\label{lem:root-location}
Let $k\geq1$ and $n\geq2k+2$, and set
\[
 a=\ceil{n/2},\qquad b=\floor{n/2},\qquad s=ab.
\]
If $\mu=\rhoA(B_{n,k})^2$, then
\begin{equation}\label{eq:root-window}
 s-2k<\mu<s-2k+1.
\end{equation}
\end{lemma}

\begin{proof}
By Proposition~\ref{prop:spectrum}, $\mu$ is one of the two roots of
\[
 f(z)=z^2-(s-2k+1)z+s-kn+k^2.
\]
The assumption $n\geq2k+2$ gives $a,b\geq k+1$.  Consequently,
\[
 f(0)=s-kn+k^2=(a-k)(b-k)>0,
\]
whereas direct substitution gives
\[
 f(1)=k(k+2-n)\leq-k^2<0.
\]
The intermediate value theorem therefore gives a root
$\eta\in(0,1)$.  Since $f$ is monic, $f(z)\to+\infty$ as
$z\to+\infty$.  Together with $f(1)<0$, this gives a second root in
$(1,\infty)$.  A quadratic has only two roots, so the root in
$(1,\infty)$ is necessarily the larger one.  Proposition~\ref{prop:spectrum}
identifies that larger root with
$\mu=\rhoA(B_{n,k})^2$.  In particular,
\[
 0<\eta<1<\mu.
\]
By Vieta's formula, the sum of the two roots is $s-2k+1$.  Hence
\[
 \mu=s-2k+1-\eta,
\]
and $0<\eta<1$ immediately yields
\[
 s-2k<\mu<s-2k+1.
\]
This is \eqref{eq:root-window}.
\end{proof}

\section{The higher-connectivity spectral theorem}

We can now state and prove the main result.

\begin{theorem}\label{thm:main}
Let $k\geq1$ and $n\geq4k+2$.  If $G$ is a triangle-free graph of order
$n$ such that $\kappa(\comp G)\geq k$, then
\begin{equation}\label{eq:main-bound}
 \rhoA(G)\leq\rhoA(B_{n,k}),
\end{equation}
where
\begin{equation}\label{eq:main-value}
 \rhoA(B_{n,k})^2
 =\frac{ab-2k+1+
 \sqrt{(ab-2k+1)^2-4(a-k)(b-k)}}{2}.
\end{equation}
Equality holds in \eqref{eq:main-bound} if and only if
$G\cong B_{n,k}$.
\end{theorem}

\begin{proof}
The assumption $n\geq4k+2$ gives $a,b>k$, so $B_{n,k}$ is well defined;
it is triangle-free, and Proposition~\ref{prop:connectivity} shows that
its complement is $k$-connected.  Thus the admissible family is nonempty.
Because there are only finitely many graphs of order $n$, we may choose
an admissible graph $G$ with maximum spectral radius.  Write
\[
 m=\e(G),\qquad s=ab=\floor{\frac{n^2}{4}}.
\]
Maximality gives $\rhoA(G)\geq\rhoA(B_{n,k})$.  The spectral Mantel
inequality in Lemma~\ref{lem:nosal} gives $m\geq\rhoA(G)^2$, while
Lemma~\ref{lem:root-location} gives
$\rhoA(B_{n,k})^2>s-2k$.  Combining these estimates yields
\[
 m\geq\rhoA(G)^2
   \geq\rhoA(B_{n,k})^2>s-2k.
\]
Since $m$ is an integer,
\begin{equation}\label{eq:m-lower}
 m\geq s-2k+1.
\end{equation}
Our order hypothesis gives $b=\floor{n/2}\geq2k+1$, and hence
\[
 m\geq s-2k+1\geq s-b+2.
\]
On the other hand, Lemma~\ref{lem:density}(ii) and
\eqref{eq:nonbip-threshold} say that every non-bipartite triangle-free
graph has at most $s-b+1$ edges.  The last display rules this out, so $G$
must be bipartite.

Fix a bipartition $P\cup Q$ of $G$, with $|P|=p$, $|Q|=q$, and
$p+q=n$.  We next show
\begin{equation}\label{eq:parts-large}
 p,q>k.
\end{equation}
Indeed, a bipartite graph with these part sizes has at most $pq$ edges,
so $m\leq pq$.  Suppose, by symmetry, that $q\leq k$.  Since
 $k<n/2$ under the order hypothesis and the function $x(n-x)$ is
 increasing on $[0,n/2]$, we obtain
 $pq=q(n-q)\leq k(n-k)$.  On the other hand, since $k(n-k)$ is an
 integer,
 \begin{align*}
 s-2k+1-k(n-k)
 &=\floor{\frac{n^2}{4}-k(n-k)}-2k+1\\
 &=\floor{\frac{(n-2k)^2}{4}}-2k+1\\
 &\geq(k+1)^2-2k+1=k^2+2>0,
 \end{align*}
Here $n-2k\geq2k+2$, so
$\floor{(n-2k)^2/4}\geq(k+1)^2$, which proves the second inequality.
Thus $k(n-k)<s-2k+1$, contradicting
$pq\geq m\geq s-2k+1$.  This proves \eqref{eq:parts-large}.

Let $F$ be the bipartite graph on $P\cup Q$ whose edges are precisely the
missing cross-edges of $G$:
\[
 \E(F)=\{xy:x\in P,\ y\in Q,\ xy\notin\E(G)\}.
\]
The cross-edges of $\comp G$ are exactly the edges of $F$, while $P$ and
$Q$ each induce a clique in $\comp G$.  We claim that $\tau(F)\geq k$.
Otherwise, $F$ has a vertex cover $S$ with $|S|\leq k-1$.  By
\eqref{eq:parts-large}, both $P\setminus S$ and $Q\setminus S$ are
nonempty.  Since $S$ meets every edge of $F$, there is no cross-edge of
$\comp G-S$ between these two nonempty sets.  They still induce cliques,
so they are distinct components of $\comp G-S$.  This contradicts the
fact that deleting fewer than $k$ vertices cannot disconnect $\comp G$.
Therefore $\tau(F)\geq k$.

By K\H{o}nig's theorem \eqref{eq:konig},
$\nu(F)=\tau(F)\geq k$, so $F$ contains a matching $M$ of size $k$.
Every edge of $M$ is absent from $G$, and all edges of $G$ are cross-edges
between $P$ and $Q$.  Consequently, as spanning graphs,
\[
 G\subseteq K_{p,q}-M\cong B_{p,q}^{(k)}.
\]
The graph $B_{p,q}^{(k)}$ is connected: because $p,q>k$, each part has an
unmatched vertex, and those two vertices join all vertices of the opposite
part.  Lemma~\ref{lem:pf} therefore gives the first inequality below, and
Proposition~\ref{prop:spectrum} gives the second:
\begin{equation}\label{eq:equality-chain}
 \rhoA(G)
 \leq\rhoA\bigl(B_{p,q}^{(k)}\bigr)
 \leq\rhoA(B_{n,k}).
\end{equation}
The leftmost term is at least the rightmost term by the maximal choice of
$G$ and the admissibility of $B_{n,k}$.  Hence equality holds throughout
\eqref{eq:equality-chain}.  If $G$ were a proper spanning subgraph of
$B_{p,q}^{(k)}$, the strict part of Lemma~\ref{lem:pf} would make the first
inequality strict.  Thus $G=B_{p,q}^{(k)}$.  Likewise, the strict balancing
statement in Proposition~\ref{prop:spectrum} makes the second inequality
an equality only when $\{p,q\}=\{a,b\}$.  Hence $G\cong B_{n,k}$.

We have proved that every spectral maximizer is $B_{n,k}$.  It follows
that every admissible graph satisfies \eqref{eq:main-bound}, and the
preceding argument also proves the asserted unique equality case.
\end{proof}

For $k=1$, the order restriction in Theorem~\ref{thm:main} can be removed
completely.

\begin{corollary}\label{cor:k-one}
Let $n\geq3$, and let $G$ be a triangle-free graph of order $n$ whose
complement is connected.  Then
\[
 \rhoA(G)\leq
 \rhoA\left(K_{\ceil{n/2},\floor{n/2}}-e\right),
\]
with equality if and only if
$G\cong K_{\ceil{n/2},\floor{n/2}}-e$.
\end{corollary}

\begin{proof}
For $k=1$, the graph $B_{n,1}$ is precisely
\[
 B_{n,1}=K_{\ceil{n/2},\floor{n/2}}-e.
\]
If $n\geq6=4\cdot1+2$, Theorem~\ref{thm:main} applies directly and gives
both the asserted inequality and its equality case.  It remains to check
$n=3,4,5$.

Suppose first that $n=3$.  If a triangle-free graph on three vertices has
at least two edges, then it is $P_3$: the only three-edge graph is the
forbidden triangle $K_3$.  However,
$\comp{P_3}=K_2\cup K_1$ is disconnected.  Therefore every admissible
graph has at most one edge.  A graph with no edges has spectral radius
zero, whereas every one-edge graph is isomorphic to $K_2\cup K_1$ and has
adjacency spectrum $\{1,0,-1\}$.  Moreover,
\[
 K_2\cup K_1\cong K_{2,1}-e=B_{3,1},
\]
and its complement is $P_3$, which is connected.  Thus $B_{3,1}$ is the
unique maximizer and its spectral radius is $1$.

Now let $n=4$.  By Lemma~\ref{lem:density}(i), $\e(G)\leq4$, with equality
only for $K_{2,2}$.  Since
$\comp{K_{2,2}}=2K_2$ is disconnected, an admissible graph must satisfy
$\e(G)\leq3$.  If $\e(G)\leq2$, Lemma~\ref{lem:nosal} gives
\[
 \rhoA(G)\leq\sqrt{\e(G)}\leq\sqrt2.
\] 
It remains to classify the case $\e(G)=3$.  A triangle-free graph on four
vertices with three edges cannot contain a cycle: a cycle with three
edges is a triangle, while a $4$-cycle requires four edges.  Hence it is
a forest.  Since a forest of order four and size three is connected, it
is a tree.  Up to isomorphism, the two trees on four vertices are $P_4$
and $K_{1,3}$.  The center of $K_{1,3}$ is isolated in its complement, so
$\comp{K_{1,3}}$ is disconnected.  On the other hand, $P_4$ is
self-complementary, and hence its complement is connected.  Thus the only
admissible three-edge graph is
\[
 P_4\cong K_{2,2}-e=B_{4,1}.
\]
The characteristic polynomial of $P_4$ is
\[
 \lambda^4-3\lambda^2+1.
\]
Its largest eigenvalue therefore satisfies
\[
 \rhoA(P_4)^2=\frac{3+\sqrt5}{2},
 \qquad
 \rhoA(P_4)=\frac{1+\sqrt5}{2}>\sqrt2.
\]
Consequently $P_4$ strictly dominates every graph with at most two edges
and is the unique maximizer for $n=4$.

Finally, let $n=5$.  Lemma~\ref{lem:density}(i) gives
$\e(G)\leq6$, with equality only for $K_{2,3}$.  Its complement is
$K_2\cup K_3$, so connectedness of $\comp G$ excludes equality and yields
$\e(G)\leq5$.  If $\e(G)\leq4$, then Lemma~\ref{lem:nosal} gives
\[
 \rhoA(G)\leq\sqrt{\e(G)}\leq2.
\]
On the other hand, $B_{5,1}=K_{2,3}-e$ is admissible: its complement
consists of a clique $K_2$ and a clique $K_3$ joined by the single
cross-edge corresponding to $e$.  Proposition~\ref{prop:spectrum}, applied
with $(p,q,k)=(2,3,1)$, gives
\[
 \rhoA(B_{5,1})^2=\frac{5+\sqrt{17}}2>4.
\]
Thus $\rhoA(B_{5,1})>2$, whereas every admissible graph with at most four
edges has spectral radius at most $2$.  Since every admissible graph has at
most five edges, an extremal graph must have exactly five edges.  We
therefore need only consider graphs with five edges.

Assume first that such a graph $G$ is non-bipartite.  It contains an odd
cycle.  Since $G$ is triangle-free and has only five vertices, a shortest
odd cycle has length five and uses every vertex.  The five edges of this
cycle already account for all the edges of $G$, so $G=C_5$.  As $C_5$ is
$2$-regular, the all-ones vector is an eigenvector with eigenvalue $2$;
the maximum-degree bound gives $\rhoA(C_5)\leq2$, and hence
$\rhoA(C_5)=2$.

Assume instead that $G$ is bipartite.  A bipartition of type $(1,4)$
allows at most four edges, so the part sizes must be $2$ and $3$.  There
are six possible cross-edges, and exactly five are present.  Hence,
uniquely up to isomorphism,
\[
 G=K_{2,3}-e=B_{5,1}.
\]
The preceding calculation shows that this graph has spectral radius
strictly greater than $2$, whereas the non-bipartite candidate $C_5$ has
spectral radius $2$.  Therefore $B_{5,1}$ is the unique maximizer.
This completes all three exceptional orders and proves the corollary.
\end{proof}

We next determine the bipartite subproblem without any large-order
assumption.  This result isolates the only remaining difficulty in the
unrestricted problem: comparison with non-bipartite triangle-free graphs.

\begin{proposition}\label{prop:bipartite-all-orders}
Let $k\geq1$ and $n\geq2k+1$.  If $G$ is a bipartite graph of order $n$
such that $\kappa(\comp G)\geq k$, then
\[
 \rhoA(G)\leq\rhoA(B_{n,k}),
\]
with equality if and only if $G\cong B_{n,k}$.
\end{proposition}

\begin{proof}
Write $P\cup Q$ for a bipartition of $G$, where
\[
 |P|=p\geq q=|Q|,
 \qquad p+q=n,
\]
and let $F$ be the bipartite graph on $P\cup Q$ whose edges are precisely
the missing cross-edges of $G$.  Thus
\[
 \E(F)=\{xy:x\in P,\ y\in Q,\ xy\notin\E(G)\}.
\]

Suppose first that $q\geq k$.  We claim that $\tau(F)\geq k$.  Otherwise,
$F$ has a vertex cover $S$ with $|S|\leq k-1$.  Since $p,q\geq k$, both
$P\setminus S$ and $Q\setminus S$ are nonempty.  Because $S$ meets every
edge of $F$, there is no edge of $F$ between these two sets and hence no
cross-edge of $\comp G-S$ between them.  Each set still induces a clique
in $\comp G-S$, so $\comp G-S$ is disconnected.  This contradicts
$\kappa(\comp G)\geq k$, and proves the claim.  By K\H{o}nig's theorem
\eqref{eq:konig}, $F$ contains a matching $M_k$ of size $k$, and therefore
\begin{equation}\label{eq:bipartite-containment}
 G\subseteq K_{p,q}-M_k\cong B_{p,q}^{(k)}
\end{equation}
as spanning graphs.

If $q>k$, then also $p>k$.  The graph on the right of
\eqref{eq:bipartite-containment} is connected, so Lemma~\ref{lem:pf} and
Proposition~\ref{prop:spectrum} give
\[
 \rhoA(G)\leq\rhoA(B_{p,q}^{(k)})\leq\rhoA(B_{n,k}).
\]
Both inequalities are strict unless, respectively, $G=B_{p,q}^{(k)}$ and
$\{p,q\}=\{\ceil{n/2},\floor{n/2}\}$.  This gives the desired inequality
and equality condition when $q>k$.

It remains within the case $q\geq k$ to consider $q=k$.  Label the deleted
matching by
\[
 M_k=\{x_i y_i:1\leq i\leq k\},
\]
where $x_i\in P$ and $y_i\in Q$.  Order the vertices of $P$ as
\[
 x_1,\ldots,x_k,x_{k+1},\ldots,x_p
\]
and those of $Q$ as $y_1,\ldots,y_k$.  With respect to these orderings, the $p$-by-$k$ bipartite adjacency matrix of $B_{p,k}^{(k)}$ is
\begin{equation}\label{eq:C-boundary}
 C=
 \begin{pmatrix}
  0&1&\cdots&1\\
  1&0&\cdots&1\\
  \vdots&\vdots&\ddots&\vdots\\
  1&1&\cdots&0\\
  1&1&\cdots&1\\
  \vdots&\vdots&&\vdots\\
  1&1&\cdots&1
 \end{pmatrix}
 =
 \begin{pmatrix}
  J_k-I_k\\
  J_{p-k,k}
 \end{pmatrix}.
\end{equation}
The first $k$ rows correspond to the matched vertices
$x_1,\ldots,x_k$.  In row $i$, the unique zero occurs in column $i$
because $x_i y_i$ is the deleted matching edge.  The remaining $p-k$
vertices of $P$ are adjacent to every vertex of $Q$, so the last $p-k$
rows are all-one rows.

Every column of $C$ contains exactly one zero and therefore has squared
norm $p-1$.  Two distinct columns have their zeros in different rows, so
they are simultaneously equal to one in exactly $p-2$ rows.  Their scalar
product is consequently $p-2$.  It follows that
\[
 C^{\mathsf T}C=
 \begin{pmatrix}
  p-1&p-2&\cdots&p-2\\
  p-2&p-1&\cdots&p-2\\
  \vdots&\vdots&\ddots&\vdots\\
  p-2&p-2&\cdots&p-1
 \end{pmatrix}
 =(p-2)J_k+I_k.
\]
The all-ones vector is an eigenvector of this matrix with eigenvalue
\[
 (p-2)k+1=pk-2k+1.
\]
All vectors orthogonal to the all-ones vector have eigenvalue $1$.
Therefore
\[
\rhoA\bigl(B_{p,k}^{(k)}\bigr)^2
 =\rho(C^{\mathsf T}C)=pk-2k+1.
\]
When $n=2k+1$, the part sizes $p=k+1$ and $q=k$ are already balanced, so
$B_{p,k}^{(k)}=B_{n,k}$.  When $n\geq2k+2$, put
$a=\ceil{n/2}$ and $b=\floor{n/2}$.  Since $a,b\geq k+1$,
\[
 ab-pk=ab-k(n-k)=(a-k)(b-k)\geq1.
\]
Lemma~\ref{lem:root-location} now yields
\[
 \rhoA(B_{n,k})^2>ab-2k
 \geq pk-2k+1=\rhoA(B_{p,k}^{(k)})^2.
\]
Thus equality is possible only at the balanced choice.  For $k\geq2$,
$B_{p,k}^{(k)}$ is connected.  Indeed, the unmatched vertex
$x_{k+1}\in P$ is adjacent to every vertex of $Q$, and each matched vertex
$x_i$ is adjacent to at least one vertex of $Q$ because $k\geq2$.
Consequently every vertex lies in the component containing $x_{k+1}$.
Lemma~\ref{lem:pf} therefore makes equality in
\eqref{eq:bipartite-containment} possible only when
 $G=B_{p,k}^{(k)}$.  If $k=1$, then $q=1$ and
\[
 B_{p,1}^{(1)}=K_{p,1}-e\cong K_{1,p-1}\cup K_1.
\]
Therefore
\[
 \rhoA\bigl(B_{p,1}^{(1)}\bigr)=\rhoA(K_{1,p-1})=\sqrt{p-1}.
\]
If the containment in \eqref{eq:bipartite-containment} is proper, then at
least one edge of the star has been deleted, since the isolated vertex has
no incident edges.  Every component of the resulting graph is then a star
with at most $p-2$ leaves or an isolated vertex.  Consequently,
\[
 \rhoA(G)\leq\sqrt{p-2}<\sqrt{p-1}
 =\rhoA\bigl(B_{p,1}^{(1)}\bigr).
\]
Hence equality is possible only when
$G=B_{p,1}^{(1)}$, which proves the required equality conclusion also for
$k=1$.

We finally exclude $q<k$.  Fix $y\in Q$ with $d_G(y)>0$ and define
\[
 S_y=(Q\setminus\{y\})\cup N_F(y).
\]
This is a vertex cover of $F$: an edge of $F$ not incident with $y$ has its
$Q$-endpoint in $Q\setminus\{y\}$, and an edge incident with $y$ has its
$P$-endpoint in $N_F(y)$.  Moreover, after deleting $S_y$ from $\comp G$,
the vertex $y$ remains, as does the nonempty set
\[
 P\setminus N_F(y)=N_G(y),
\]
but $y$ has no neighbor in this set in $\comp G$.  Hence $\comp G-S_y$ is
disconnected.  The connectivity hypothesis forces
\[
 |S_y|=q-1+d_F(y)\geq k,
\]
and therefore
\[
 d_G(y)=p-d_F(y)\leq p-(k-q+1)=n-k-1.
\]
The same degree bound is trivial when $d_G(y)=0$.  Summing over $Q$ gives
\begin{equation}\label{eq:small-part-edge-bound}
 \e(G)=\sum_{y\in Q}d_G(y)
 \leq q(n-k-1)
 \leq(k-1)(n-k-1)=:E.
\end{equation}

It remains to compare this number with $\rhoA(B_{n,k})^2$.  Let
$s=\floor{n^2/4}$ and $b=\floor{n/2}$.

First suppose that $n=2b+1$ and put $d=b-k\geq0$.  Then
\[
 s=\floor{\frac{(2b+1)^2}{4}}=b(b+1)
 \qquad\text{and}\qquad
 E=(k-1)(2b-k).
\]
Consequently,
\begin{align*}
 s-2k-E
 &=b(b+1)-2k-(k-1)(2b-k)\\
 &=b^2+b-2k-\bigl(2bk-k^2-2b+k\bigr)\\
 &=(b-k)^2+3(b-k)=d^2+3d.
\end{align*}
If $d=0$, then $n=2k+1$, $E=k(k-1)$, and the calculation for
$B_{k+1,k}^{(k)}$ above gives
\[
 \rhoA(B_{n,k})^2=k(k+1)-2k+1=E+1.
\]
If $d\geq1$, then $s-2k-E=d^2+3d>0$; hence
Lemma~\ref{lem:root-location} gives
\[
 \rhoA(B_{n,k})^2>s-2k>E.
\]

Now suppose that $n=2b$.  Since $n\geq2k+1$, we have $b\geq k+1$; put
$d=b-k\geq1$.  In this case,
\[
 s=b^2,
 \qquad E=(k-1)(2b-k-1),
\]
and
\begin{align*}
 s-2k-E
 &=b^2-2k-(k-1)(2b-k-1)\\
 &=b^2-2bk+k^2+2b-2k-1\\
 &=(b-k)^2+2(b-k)-1=d^2+2d-1>0.
\end{align*}
The same lemma therefore gives
$\rhoA(B_{n,k})^2>s-2k>E$.  Thus, in both parity cases, Nosal's inequality
and \eqref{eq:small-part-edge-bound} imply
\[
 \rhoA(G)^2\leq \e(G)\leq E<\rhoA(B_{n,k})^2.
\]
Thus $q<k$ cannot give equality.

For completeness, $B_{n,k}$ is admissible in the stated range.  Its
complement consists of two cliques joined by the deleted $k$-matching.
Deleting fewer than $k$ vertices leaves both cliques nonempty and leaves
at least one matching edge, so the remaining graph is connected.  This
proves $\kappa(\comp{B_{n,k}})\geq k$ and completes the proof.
\end{proof}

At the boundary order $n=2k$, the extremal value is also elementary, but
uniqueness need not hold.

\begin{proposition}\label{prop:boundary-order}
For every $k\geq2$,
\[
 \operatorname{spex}_{\kappa}(2k,K_3;k)=k-1.
\]
Among the admissible graphs, equality holds precisely for those containing
a $(k-1)$-regular connected component.  In particular,
$B_{2k,k}=K_{k,k}-M_k$ is extremal.
\end{proposition}

\begin{proof}
Let $G$ be admissible and put $H=\comp G$.  The inequality
$\kappa(H)\geq k$ implies $\delta(H)\geq k$.  Since $|V(H)|=2k$, for every
vertex $v$ we have
\[
 d_G(v)=2k-1-d_H(v)\leq k-1.
\]
The standard maximum-degree bound therefore gives
$\rhoA(G)\leq\Delta(G)\leq k-1$.

The graph $B_{2k,k}$ is $(k-1)$-regular, so its spectral radius is $k-1$.
Its complement $H_0$ consists of two copies of $K_k$, say on
$\{x_1,\ldots,x_k\}$ and $\{y_1,\ldots,y_k\}$, joined by the perfect
matching $x_i y_i$, $1\leq i\leq k$.  If fewer than $k$ vertices are
deleted, both cliques retain a vertex and at least one matching edge has
both endpoints left; the surviving matching edge joins the two surviving
cliques.  Hence $\kappa(H_0)\geq k$.  Conversely, deleting
\[
 \{x_1,\ldots,x_{k-1},y_k\}
\]
leaves $x_k$ separated from $\{y_1,\ldots,y_{k-1}\}$, so
$\kappa(H_0)\leq k$.  Thus $B_{2k,k}$ is admissible and attains the upper
bound.

It remains to justify the equality description.  Every component $D$ of
$G$ has maximum degree at most $k-1$.  If $D$ is connected, then
$\rhoA(D)\leq k-1$, with equality if and only if $D$ is
$(k-1)$-regular: indeed, for a positive Perron vector, an entry of maximum
value satisfies
\[
 \rhoA(D)x_v=\sum_{u\in N_D(v)}x_u\leq d_D(v)x_v\leq(k-1)x_v,
\]
and equality forces every neighbor to have the same maximum entry;
connectedness propagates this equality to all vertices and forces every
degree to be $k-1$.  Since the spectral radius of a disconnected graph is
the maximum of the spectral radii of its components, $\rhoA(G)=k-1$ holds
exactly when one component is $(k-1)$-regular.
\end{proof}

We now turn to the first nontrivial connectivity parameter.  For $n\geq5$
let $S_n$ denote the graph obtained by subdividing one edge of
\[
 K_{r,t},\qquad
 r=\floor{(n-1)/2},\quad t=\ceil{(n-1)/2}.
\]
Thus $S_n$ has order $n$, is triangle-free, and is non-bipartite.
The following lemma combines the order-based spectral extremal theorem in
\cite{LinNingWu2021} with a direct calculation of the connectivity of the
complement of its extremal graph.

\begin{lemma}\label{lem:lnw-order}
If $n\geq5$ and $G$ is a non-bipartite triangle-free graph of order $n$,
then
\[
 \rhoA(G)\leq\rhoA(S_n),
\]
with equality if and only if $G\cong S_n$.  Moreover,
$\kappa(\comp{S_n})=2$.
\end{lemma}

\begin{proof}
The spectral inequality and its equality case are the order-based extremal
theorem of Lin, Ning, and Wu \cite{LinNingWu2021}.  We prove the additional
connectivity assertion, which is needed here.

Let $A,B$ be the parts of $K_{r,t}$, let $uv$ be the subdivided edge with
$u\in A$ and $v\in B$, and let $w$ be the new subdivision vertex.  In
$\comp{S_n}$, both $A$ and $B$ induce cliques, the only edge between these
two cliques is $uv$, and
\[
 N_{\comp{S_n}}(w)=(A\setminus\{u\})\cup(B\setminus\{v\}).
\]
Here $r,t\geq2$.  Removing $w$ leaves the two cliques joined by $uv$.
Removing $u$ or $v$ leaves $w$ joining the nonempty remainders of the two
cliques, and removing any other vertex leaves both the edge $uv$ and the
required connections through $w$.  Thus no single vertex disconnects
$\comp{S_n}$, and $\kappa(\comp{S_n})\geq2$.  On the other hand, after
deleting $u$ and $w$, the nonempty sets $A\setminus\{u\}$ and $B$ are two
cliques with no edge between them.  Hence $\{u,w\}$ is a vertex cut and
$\kappa(\comp{S_n})=2$.
\end{proof}

\begin{theorem}\label{thm:k-two}
For every admissible order $n\geq3$,
\[
 \operatorname{spex}_{\kappa}(n,K_3;2)=
 \begin{cases}
  0,&n=3,\\
  1,&n=4,\\
  2,&n=5,\\
  \rhoA(B_{n,2}),&n\geq6.
 \end{cases}
\]
For $n=3$ the unique extremal graph is the empty graph.  For $n=4$ the
extremal graphs are $K_2\cup2K_1$ and $2K_2$.  For $n=5$ the unique
extremal graph is $C_5$, and for every $n\geq6$ the unique extremal graph
is $B_{n,2}$.  Consequently, $n_0(2)=6$.
\end{theorem}

\begin{proof}
For $n=3$, the condition $\kappa(\comp G)\geq2$ forces
$\comp G=K_3$, so $G$ is empty.  For $n=4$, the minimum-degree inequality
$\delta(\comp G)\geq2$ gives $\Delta(G)\leq1$.  Thus $G$ is a matching and
$\rhoA(G)\leq1$.  Both nonempty possibilities
$K_2\cup2K_1$ and $2K_2$ attain this value; their complements are,
respectively, $K_4-e$ and $C_4$, both of which are $2$-connected.  This
proves the first two cases.

Let $n\geq5$.  Proposition~\ref{prop:bipartite-all-orders} says that the
unique bipartite maximizer is $B_{n,2}$.  Lemma~\ref{lem:lnw-order} says
that the unique non-bipartite maximizer is $S_n$, and its complement is
$2$-connected.  It therefore remains only to compare these two candidates.
For $n=5$, one has
\[
 S_5=C_5,\qquad \rhoA(S_5)=2,
 \qquad B_{5,2}\cong P_5,\qquad \rhoA(B_{5,2})=\sqrt3.
\]
Thus $C_5$ is the unique extremal graph at order five.

We first perform the comparison for $n\geq8$.  Put
$s=\floor{n^2/4}$ and $b=\floor{n/2}$.  Subdividing an edge replaces one edge by two, so \eqref{eq:nonbip-threshold} gives
\[
 \e(S_n)=\floor{\frac{(n-1)^2}{4}}+1=s-b+1.
\]
The fixed-size theorem of Lin, Ning, and Wu
\cite{LinNingWu2021} yields
\[
 \rhoA(S_n)^2\leq \e(S_n)-1=s-b.
\]
Since $b\geq4$, Lemma~\ref{lem:root-location}, applied with $k=2$, gives
\[
 \rhoA(B_{n,2})^2>s-4\geq s-b\geq\rhoA(S_n)^2.
\]
Hence $B_{n,2}$ strictly dominates $S_n$ for every $n\geq8$.

For clarity, we give exact comparisons for the two remaining orders.  In $S_n$, retain the notation $u,v,w,A\setminus\{u\},B\setminus\{v\}$ from
the proof of Lemma~\ref{lem:lnw-order}.  These five cells form an equitable partition, with quotient matrix
\begin{equation}\label{eq:Sn-quotient}
 R_{r,t}=\begin{pmatrix}
 0&0&1&0&t-1\\
 0&0&1&r-1&0\\
 1&1&0&0&0\\
 0&1&0&0&t-1\\
 1&0&0&r-1&0
 \end{pmatrix}.
\end{equation}
Indeed, the entries in each row are the numbers of neighbors that a vertex in the corresponding cell has in the five cells.  The graph $S_n$ is connected, and permutations inside either of the last two cells are graph automorphisms.  Uniqueness of the positive Perron vector therefore makes it constant on every cell.  Consequently $\rhoA(S_n)$ is the Perron root of $R_{r,t}$.

For $n=6$, we have $(r,t)=(2,3)$.  Expanding
$\det(xI-R_{2,3})$ gives
\[
 f_6(x)=x^5-7x^3+9x-4.
\]
The graph $S_6$ contains a $5$-cycle, so $\rhoA(S_6)\geq2$.  Moreover,
\[
 f_6'(x)=5x^4-21x^2+9>0\qquad(x\geq2):
\]
after putting $y=x^2\geq4$, the expression $5y^2-21y+9$ has value $5$ at
$y=4$ and is strictly increasing thereon.  Proposition~\ref{prop:spectrum}
gives $\rhoA(B_{6,2})=1+\sqrt2$, and direct substitution gives
\[
 f_6(1+\sqrt2)=3(\sqrt2-1)>0.
\]
Since $f_6(\rhoA(S_6))=0$ and $f_6$ is strictly increasing on the relevant
interval, $\rhoA(S_6)<\rhoA(B_{6,2})$.

For $n=7$, we have $(r,t)=(3,3)$ and
\[
 f_7(x):=\det(xI-R_{3,3})=x^5-10x^3+16x-8.
\]
The graph $S_7$ has ten edges.  The Rayleigh quotient of the all-ones
vector therefore gives
\[
 \rhoA(S_7)\geq\frac{2e(S_7)}7=\frac{20}7.
\]
Now $f_7'(x)=5x^4-30x^2+16$.  Its larger zero as a polynomial in $x^2$
is
\[
 3+\frac{\sqrt{145}}5<\left(\frac{20}7\right)^2,
\]
so $f_7$ is strictly increasing for $x\geq20/7$.  Put
$\mu=\rhoA(B_{7,2})^2$.  Proposition~\ref{prop:spectrum} gives
\[
 \mu^2-9\mu+2=0,
 \qquad \mu=\frac{9+\sqrt{73}}2.
\]
In particular, $8.5<\mu<9$, and hence
\begin{align*}
 f_7(\sqrt\mu)
 &=\sqrt\mu\,(\mu^2-10\mu+16)-8\\
 &=\sqrt\mu\,(14-\mu)-8>0.
\end{align*}
The last inequality follows already from $\sqrt\mu>2$ and
$14-\mu>5$.  Since both relevant numbers lie in the interval of strict
increase and $f_7(\rhoA(S_7))=0$, we obtain
$\rhoA(S_7)<\rhoA(B_{7,2})$.

For every $n\geq6$, therefore, the bipartite candidate strictly dominates
the unique non-bipartite candidate.  Proposition~\ref{prop:bipartite-all-orders}
then supplies uniqueness of $B_{n,2}$ and completes the proof.
\end{proof}

The exact expression also quantifies the spectral cost of imposing
$k$-connectivity on the complement.

\begin{corollary}\label{cor:asymptotic}
For fixed $k$ and $n\to\infty$, put $s=\floor{n^2/4}$.  Then
\[
 \rhoA(B_{n,k})
 =\sqrt{s}-\frac{k}{\sqrt{s}}+O_k(n^{-2}).
\]
In particular,
\[
 \rhoA(B_{n,k})=
 \begin{cases}
 \displaystyle \frac n2-\frac{2k}{n}+O_k(n^{-2}),& n\text{ even},\\[6pt]
 \displaystyle \frac n2-\frac{2k+1/4}{n}+O_k(n^{-2}),& n\text{ odd}.
 \end{cases}
\]
\end{corollary}

\begin{proof}
We start by fixing $k$.  For all sufficiently large $n$, the hypotheses of Lemma~\ref{lem:root-location} hold.  Let $\eta\in(0,1)$ be the smaller root of the quadratic
\begin{equation}\label{eq:asymptotic-quadratic}
 z^2-(s-2k+1)z+s-kn+k^2=0,
\end{equation}
and let
\[
 \mu=\rhoA(B_{n,k})^2
\]
be its larger root.  Since $0<\eta<1$, we may write
\[
 \eta=1-\varepsilon, \qquad 0<\varepsilon<1.
\]
Substituting $z=1-\varepsilon$ into \eqref{eq:asymptotic-quadratic} gives
\begin{align*}
0
 &=(1-\varepsilon)^2
   -(s-2k+1)(1-\varepsilon)+s-kn+k^2\\
 &=\bigl(1-(s-2k+1)+s-kn+k^2\bigr)
   +\bigl(s-2k-1\bigr)\varepsilon+\varepsilon^2\\
 &=-k(n-k-2)+(s-2k-1)\varepsilon+\varepsilon^2.
\end{align*}
Therefore
\begin{equation}\label{eq:epsilon-identity}
 (s-2k-1)\varepsilon+\varepsilon^2=k(n-k-2).
\end{equation}

Because $s=\floor{n^2/4}$, we have $s-2k-1=\Theta(n^2)$ for fixed $k$,
whereas $k(n-k-2)=O_k(n)$.  For sufficiently large $n$ the coefficient
$s-2k-1$ is positive.  Dropping the nonnegative term
$\varepsilon^2$ from the left-hand side of
\eqref{eq:epsilon-identity} yields
\[
 0<\varepsilon
 \leq\frac{k(n-k-2)}{s-2k-1}
 =O_k(n^{-1}).
\]
Thus the smaller root differs from $1$ by only $O_k(n^{-1})$.

By Vieta's formula,
\[
 \mu+\eta=s-2k+1.
\]
Using $\eta=1-\varepsilon$, we obtain the exact relation
\[
 \mu=s-2k+\varepsilon.
\]
Set $\delta=-2k+\varepsilon$.  Then $\delta=O_k(1)$ and
\[
 \rhoA(B_{n,k})=\sqrt{s+\delta}.
\]
Taylor's formula at $s$ gives
\begin{align*}
 \sqrt{s+\delta}
 &=\sqrt{s}+\frac{\delta}{2\sqrt{s}}
   +O_k\left(\frac{\delta^2}{s^{3/2}}\right)\\
 &=\sqrt{s}-\frac{k}{\sqrt{s}}
   +\frac{\varepsilon}{2\sqrt{s}}+O_k(n^{-3}).
\end{align*}
Here $s=\Theta(n^2)$ and $\varepsilon=O_k(n^{-1})$, so
\[
 \frac{\varepsilon}{2\sqrt{s}}=O_k(n^{-2}).
\]
Consequently,
\[
 \rhoA(B_{n,k})
 =\sqrt{s}-\frac{k}{\sqrt{s}}+O_k(n^{-2}).
\]

It remains to express this expansion in terms of $n$.  If $n$ is even,
then $s=n^2/4$, and therefore
\[
 \sqrt{s}=\frac n2, \qquad \frac{k}{\sqrt{s}}=\frac{2k}{n}.
\]
This gives
\[
 \rhoA(B_{n,k})=\frac n2-\frac{2k}{n}+O_k(n^{-2}).
\]
If $n$ is odd, then $s=(n^2-1)/4$.  Using the binomial expansions
$(1-x)^{1/2}=1-x/2+O(x^2)$ and
$(1-x)^{-1/2}=1+x/2+O(x^2)$ with $x=n^{-2}$, we obtain
\begin{align*}
 \sqrt{s}
 &=\frac n2\sqrt{1-\frac1{n^2}}
   =\frac n2-\frac1{4n}+O(n^{-3}),\\
 \frac1{\sqrt{s}}
 &=\frac2n\left(1-\frac1{n^2}\right)^{-1/2}
   =\frac2n+O(n^{-3}).
\end{align*}
Hence
\begin{align*}
 \rhoA(B_{n,k})
 &=\left(\frac n2-\frac1{4n}+O(n^{-3})\right)
   -k\left(\frac2n+O(n^{-3})\right)+O_k(n^{-2})\\
 &=\frac n2-\frac{2k+1/4}{n}+O_k(n^{-2}).
\end{align*}
This proves both parity-specific formulas.
\end{proof}

\section{Concluding remarks and open problems}

Theorem~\ref{thm:main} gives a higher-connectivity spectral counterpart to
the $\alpha=2$ edge-extremal result of Bougard and Joret.  The extremal
construction has a transparent dual description: its complement consists
of two balanced cliques joined by a matching of size $k$.  The proof also
explains why a matching is forced.  Once bipartiteness is known, a vertex
cover of fewer than $k$ missing cross-edges would be a vertex cut of the
complement, and K\H{o}nig's theorem converts this obstruction into the
desired matching.  Proposition~\ref{prop:bipartite-all-orders} shows that
this structural and spectral part of the argument already works for every
$n\geq2k+1$; the large-order hypothesis is needed only to exclude
non-bipartite competitors.

The hypothesis $n\geq4k+2$ enters only in the density-to-bipartiteness
step.  More precisely, it ensures that the spectral lower bound supplied
by $B_{n,k}$ crosses the sharp non-bipartite edge threshold in
Lemma~\ref{lem:density}(ii).  Indeed, the proof gives
$\e(G)\geq s-2k+1$, whereas a non-bipartite triangle-free graph has at
most $s-b+1$ edges, where
$s=\floor{n^2/4}$ and $b=\floor{n/2}$.  The former bound excludes the
latter precisely when
\[
 s-2k+1\geq s-b+2,
\]
or equivalently $b\geq2k+1$.  Thus $n\geq4k+2$ is the exact threshold
for the present proof method; we do not claim that it is the best possible
threshold for the theorem itself.

The natural minimum-order range cannot simply be included without
exceptions.  For example, take $(n,k)=(5,2)$.  Then
$\comp{C_5}\cong C_5$ is $2$-connected and $\rhoA(C_5)=2$, while
\[
 B_{5,2}=K_{3,2}-M_2\cong P_5,
 \qquad \rhoA(B_{5,2})=\rhoA(P_5)=\sqrt3.
\]
Hence $B_{5,2}$ is not extremal.  Uniqueness can also fail at the boundary
order $n=2k$: for $(n,k)=(6,3)$, one has
$B_{6,3}=K_{3,3}-M_3\cong C_6$ and $\rhoA(B_{6,3})=2$, whereas
$C_5\cup K_1$ also has spectral radius $2$ and its complement
$K_1\vee C_5$ is $3$-connected.  Indeed, deletion of at most two
vertices leaves this join connected: if its universal vertex survives,
it joins all remaining vertices, while deleting it and at most one
cycle vertex leaves a connected path or cycle.  Since every cycle vertex
has degree three, its connectivity is exactly three.  Theorem~\ref{thm:k-two}
shows that the first obstruction is sharp, namely $n_0(2)=6$, while the
second gives $n_0(3)\geq7$.  These examples make the following remaining
exact-range problem natural.

\begin{problem}\label{prob:range}
For every $k\geq3$, determine $\operatorname{spex}_{\kappa}(n,K_3;k)$ at
the orders not covered by Theorem~\ref{thm:main} and
Proposition~\ref{prop:boundary-order}, and determine the least $n_0(k)$
such that $B_{n,k}$ is the unique extremal graph for every
$n\geq n_0(k)$.
\end{problem}

Corollary~\ref{cor:k-one} gives $n_0(1)=3$, and
Theorem~\ref{thm:k-two} gives $n_0(2)=6$.  Proposition~\ref{prop:boundary-order}
also determines the value at the boundary order $n=2k$ for every $k\geq2$,
while Proposition~\ref{prop:bipartite-all-orders} completely disposes of
all bipartite competitors for $n\geq2k+1$.  Thus for $k\geq3$ the remaining
issue is to compare the balanced bipartite candidate with dense
non-bipartite triangle-free graphs whose complements retain
$k$-connectivity.  The order-based theorem of Lin, Ning, and Wu
\cite{LinNingWu2021} identifies the unrestricted non-bipartite spectral
maximizer, but for $k\geq3$ its complement is only $2$-connected; a
connectivity-sensitive refinement appears to be necessary.

The next direction is to replace triangles by larger cliques.  If
$r\geq3$, a $K_{r+1}$-free spectral extremizer is expected first to become
$r$-partite and then to be obtained from a nearly balanced complete
$r$-partite graph by deleting a sparse cross-edge configuration whose
complement is $k$-connected.  Unlike the case $r=2$, the auxiliary missing-edge graph now links $r$ cliques, and neither the optimal linkage nor its
spectral placement is immediate.

\begin{problem}\label{prob:general-r}
For fixed $r,k\geq2$ and sufficiently large $n$, determine the maximum
spectral radius of an $n$-vertex $K_{r+1}$-free graph $G$ satisfying
$\kappa(\comp G)\geq k$, and characterize all equality cases.
\end{problem}

One may also replace vertex-connectivity by edge-connectivity.  In the bipartite reduction, edge-connectivity controls the number, rather than the vertex-cover number, of missing cross-edges.  This suggests that the spectrally optimal deleted set may be concentrated rather than matched, leading to a different matrix optimization problem.  Finally, signless Laplacian and normalized adjacency versions of Problems \ref{prob:range}--\ref{prob:general-r} appear to be unexplored.

\section*{Declaration of generative AI and AI-assisted technologies in the manuscript preparation process}

During the preparation of this work, the authors used OpenAI Codex to discuss proof strategies, organize and check bibliographic
information, check algebraic calculations and proof exposition, and improve language and readability.  After using this tool, the authors
reviewed and edited the content as needed, independently verified the cited sources, calculations, statements, and proofs, and take full
responsibility for the content of the article.

\section*{Declaration of competing interest}

The authors declare that they have no known competing financial interests or personal relationships that could have appeared to influence the work reported in this article.


\section*{Data availability}

Data sharing is not applicable to this article as no datasets were generated or analyzed during the current study.

\end{document}